\documentclass[letterpaper, 10 pt, conference]{ieeeconf}  % Comment this line out if you need a4paper

\IEEEoverridecommandlockouts                              % This command is only needed if 
\usepackage{url}
\usepackage{array,tabularx}
\usepackage{multicol, nonfloat}  %% for two columns. NEVER REMOVE!!
\usepackage{epsfig} % for postscript graphics files
\usepackage{amssymb}  % assumes amsmath package installed
\newtheorem{theorem}{Theorem}[section]
\newtheorem{corollary}[theorem]{Corollary}
\newtheorem{lemma}[theorem]{Lemma}

\newtheorem{definition}{Definition}[section]

\usepackage{blindtext}
\usepackage[utf8]{inputenc}	
\usepackage[english]{babel}
\usepackage{graphicx}
\usepackage{amsmath}
\usepackage{amsfonts}
\usepackage{epstopdf}
\usepackage{cite}
\usepackage{microtype}

\newtheorem{statement}{Statement}

\newcommand\numberthis{\addtocounter{equation}{1}\tag{\theequation}}
\newcommand{\quotes}[1]{``#1''}

\title{\LARGE \bf
A Graph-Theoretic Analysis of Distributed Replicator Dynamic}

\author{Prashant Bansode, Aniket Deshpande, Navdeep Singh% <-this % stops a space
\thanks{Prashant Bansode is with Department of Instrumentation Engineering, Ramrao Adik Institute of Technology, Mumbai, 400706 India.~{\tt\small prashant.bansode@rait.ac.in}}
\thanks{Aniket Deshpande and Navdeep Singh are with Department of Electrical Engineering, Veermata Jijabai Technological Institute, Mumbai, 400019 India.}
}
\begin{document}

\maketitle
\thispagestyle{empty}
\pagestyle{empty}

%%%%%%%%%%%%%%%%%%%%%%%%%%%%%%%%%%%%%%%%%%%%%%%%%%%%%%%%%%%%%%%%%%%%%%%%%%%%%%%%
\begin{abstract}
This paper attempts to develop a graph-theoretic multi-agent perspective of population games to study the \quotes{truncation} behavior. The proposed method considers fitness of the population	as a dynamical system to address the issue of restrictive description of this behavior which pertains to the underlying population dynamic. The fitness dynamic resembles an agreement protocol that enables comments on the steady-state characteristics of the graph that represents the population structure. The structural attributes of the underlying graph and the truncation behavior are emphasized by exploiting the spectral properties of the associated Laplacian matrix. The asymptotic stability of the fitness agreement protocol has been shown to be sufficient in concluding the stability of population dynamics. Simulation results validating the proposed hypothesis have been discussed.
\end{abstract}

%%%%%%%%%%%%%%%%%%%%%\textbf{Introduction}%%%%%%%%%%%%%%%%%%%%%%%%%%%%%%%%%%%%%%%%%%%%%%%%%%%%%%%%%%%
\section{Introduction}
\indent Evolutionary game theory has emerged from behavioral ecology, a field that formally analyzes the influence of ecological factors on the evolution of different species or populations in their behavioral contexts. Their evolution pertains to how they revise strategies under certain protocols. The central idea of evolutionary games is to model the evolution of the Nash equilibrium (NE) of strategic form games using a set of differential equations that define the strategy revision protocol. The population games belong to a class of evolutionary games which model strategic interactions among the populations with a key assumption that these populations only occasionally receive opportunities to revise their strategies. It is apropos to consider \cite{von2007theory,nash1951non,basar1999dynamic,narahari2014game,haurie2012games} for a comprehensive literature survey on game theory. The literature on evolutionary games, population games and various revision protocols is spanned over last four decades \cite{smith1974theory,taylor1978evolutionary,weibull1997evolutionary,hofbauer2003evolutionary,hofbauer2009stable,sandholm2015population}; surveying this literature is beyond the scope of this paper. 

\subsection{Relevant literature}
Population games have lately inspired a wide range of the distributed optimization problem-solving techniques in multi-agent systems. A decently recent work on the application of these games to multi-agent systems includes \cite{ marden2009cooperative,ramirez2010population,pantoja2012distributed,pantoja2011dispatch,barreiro2014constrained,mojica2014dynamic,obando2014building, 7172156,barreiro2016distributed}. An application of population games to a certain class of cooperative control problems is demonstrated in \cite{ marden2009cooperative}. Most of the distributed optimization problems in multi-agent systems can be pursued as dynamic resource allocation problems. To this end, researchers have sought to exploit the distributed version of a strategy revision protocol known as distributed replicator dynamic (DRD) \cite{ ramirez2010population,pantoja2012distributed,pantoja2011dispatch,barreiro2014constrained,mojica2014dynamic,obando2014building,7172156,barreiro2016distributed}. The DRD overcomes the drawback of classic replicator dynamic by considering only local-level mutual interactions among the neighboring agents. To elaborate further, \cite{pantoja2012distributed} demonstrates the application of the DRD to control the luminance of several lighting zones with electrical power as a resource under consideration. A similar but more appealing approach is considered in \cite{pantoja2011dispatch} for optimal dispatch of distributed generators in a microgrid. 

\subsection{Motivation}
Authors' interest in the graph-theoretic analysis of population games is motivated by observation of the \quotes{truncation} behavior of agents \cite{pantoja2011dispatch}. The truncation behavior occurs when the population-share associated with the agent has fitness below the average fitness at equilibrium. The fitness may be regarded as the payoff an agent receives for playing a particular strategy\cite{sandholm2010population}. If one considers a graph that depicts the structure of a finite population in the game, then as the population-shares associated with different agents undergo strategy revisions, the interconnection structure of the graph gets significantly influenced. For example, the population-share associated with an agent, can render it fit or unfit for survival, based on its strategy being profitable or nonprofitable as observed  by the revision protocol.\cite{sandholm2010population}. This scenario is theoretically explained in \cite{pantoja2011dispatch}. However, it appears that the subject still lacks a detailed mathematical analysis. 

In population games, an agent may refer to a strategy shared by a subset of the total population. However, in context of the reported work, it is referred to as the fitness associated with the corresponding subset of the population. The population dynamic (distributed replicator dynamic in this case) does not describe the evolution of fitness of the agents. Hence it is insufficient for understanding how the agents converge to the agreement value. It only uses the fitness function as an argument and describes the evolution of the population state. As truncation is the behavior an agent may exhibit at a steady-state, it is more appropriate to define the fitness dynamic to understand how this behavior emerges. The proposed dynamic describes the evolution of the fitness values. The approach of the reported work provides a graph-theoretic explanation to the truncation behavior of the agents with a few more contributions stated below.
\subsection{Contributions}
\indent  The main objective of this paper is to seek an analytical rationale of the \quotes{truncation} behavior of the agents using spectral properties of the graph theory. To the best of authors' knowledge, no significant contribution is seen towards analyzing this behavior in population games by exploiting rich spectral properties of the interaction graph of the game. The contributions of the present work as follows: 
\begin{enumerate}
	\item The formulation of the fitness dynamic of the agents and its realization as a linear parameter-varying dynamic agreement protocol that allows analyzing these games from the perspective of the multi-agent systems. 
	\item Analysis of truncation behavior of the agents using spectral properties of the graph-Laplacian matrix.
	\item A graph-theoretic approach to ensure the stability of the Nash equilibrium of the potential games from a graph-theoretic perspective.
\end{enumerate}
\subsection{Organization of the paper}
\indent The rest of the paper is organized as follows: Some preliminaries on graph theory and potential games are presented in Section \ref{sec2}. Section \ref{sec3} explains the network model of the distributed replicator dynamic and the fitness agreement protocol. The sufficient conditions for the existence of solutions of the distributed replicator dynamic and the fitness dynamic are presented in Subsection \ref{sec4}. Subsection \ref{sect5} and Subsection \ref{sec5.b} elaborate on the existence of Nash equilibrium and its stability, respectively. Section \ref{sim} presents the simulation examples while Section \ref{con} concludes the paper. 
%%%%%%%%%%%%%%%%%%%%%%%%%%%%%%%%%%%%%%%%%%%%%%%%%%%%%%%%%%%%%%%%%%%%%%%%%%%%%%%%%%%%%%%%
%%%%%%%%%%%%%%%%%%%%%%	      Preliminaries			%%%%%%%%%%%%%%%%%%%%%%%%%%%%%%%%%%%%
%%%%%%%%%%%%%%%%%%%%%%%%%%%%%%%%%%%%%%%%%%%%%%%%%%%%%%%%%%%%%%%%%%%%%%%%%%%%%%%%%%%%%%%% 
\section{Preliminaries}\label{sec2}
%%%%%%%%%%%%%%%%%%%%%%%%%%%%%%%%%%%%%%%% Graph theory %%%%%%%%%%%%%%%%%%%%%%%%%%%%%%%%%%
\subsection{Graph theory}\label{sec2.a}
An (undirected) graph $\mathcal{G}$ is defined by the couple $\mathcal{G}=(\mathcal{N}, \mathcal{E})$ with $n$ nodes (or agents) $\mathcal{N}= \{1, 2, \ldots, n\}$ and a number of edges in the set $\mathcal{E} \subseteq \mathcal{N} \times \mathcal{N}$. The neighbor set of the $i^{\rm th}$ node is $\mathcal{N}_i=\{j\in \mathcal{N}| (j,i) \in \mathcal{E} \}$, where $i\in\{1,2,\ldots,n\}$. The number of nodes~$n$ is the cardinality of the graph.
%%%%%%%%%%%%%%%%%%%%%%%%%%%%%%%%%%%%%%%% Population games %%%%%%%%%%%%%%%%%%%%%%%%%%%%%%
\subsection{Population games}\label{sec2.b}

A population game is characterized by a continuous vector-valued payoff function $f = (f_1,\dots,f_n) :\Re^n_+ \rightarrow \Re^n$. The agents of the population game choose their strategies from a pure strategy set $\mathcal{S}=\{1,2,\ldots,n\}$ and their aggregate behavior is described as a population state $p(t) = \{p_1(t), \ldots, p_n(t)\}$ that belongs to a simplex $\bigtriangleup$ for all $t\geq 0$ such that
\begin{equation}
\bigtriangleup=\{p(t) \in \Re^n_+:\sum_{i=1}^{n}p_i(t)=P_{\mathrm{tot}}\}, \label{simplex}
\end{equation} 
where $P_{\mathrm{tot}}$ is the total population of the game\cite{sandholm2010population}. For the sake of notational simplicity, here onwards the dependence on $t$ is dropped wherever it is obvious. The scalar $p_i$ denotes the number of agents choosing the pure strategy $i$ and $f_i$ denotes the payoff to strategy $i$. 

\indent Given a population game $f$, a Nash equilibrium $p^*$ of $f$ is defined as given below\cite{barreiro2016distributed,sandholm2010population}:
\begin{definition}\label{nef}
	Given a population game $f$, the set of Nash equilibria is defined as follows:
	\begin{align}
	\begin{split}
	NE(f) = \{p\in\bigtriangleup:p_i>0\implies f_i(p)\geq f_j(p),\\
	~\forall i,j \in \mathcal{S}\}. \label{nepsi}
	\end{split}
	\end{align}	 
\end{definition}	
Further, $f$ is a {\em potential game\/} if there exists a continuously differentiable potential function $V:\Re^n_+ \rightarrow \Re$ such that $f(p)=\nabla V(p)$. It is assumed that $V(p)$ is strictly quadratic concave, i.e. $V(p)=-p^T H p + b^Tp +c$ with $H>0$. We replace \quotes{payoff} by \quotes{fitness} throughout.
%%%%%%%%%%%%%%%%%%%%%%%%%%%%%%%%%%%%%%%%%%%%%%%%%%%%%%%%%%%%%%%%%%%%%%%%%%%%%%%%%%%%%%%%
%%%%%%%%%%%%%%%%%%%%%%	      Problem Formulation	%%%%%%%%%%%%%%%%%%%%%%%%%%%%%%%%%%%%
%%%%%%%%%%%%%%%%%%%%%%%%%%%%%%%%%%%%%%%%%%%%%%%%%%%%%%%%%%%%%%%%%%%%%%%%%%%%%%%%%%%%%%%% 
\section{Problem Formulation}\label{sec3}
\indent This section first introduces the classic and distributed versions of the replicator dynamic and proceeds towards the problem formulation and main results of the paper.
%%%%%%%%%%%%%%%%%%%%%%	      Replicator Dynamic	%%%%%%%%%%%%%%%%%%%%%%%%%%%%%%%%%%%%
\subsection{Replicator dynamic}
The replicator dynamic is a strategy revision protocol which describes the aggregate behavior of the agents through the evolution of the population state $p(t)$. Assuming that there exists a network of agents such that each agent receives a certain share of the total population $p$, then the classic replicator dynamic describes how the population-shares associated with these agents vary with respect to time. It is given by 
\begin{equation}
\dot p_i (t) = p_i(t) \, [f_i(p(t))-\bar{f}(p(t))],\label{cp}
\end{equation}
where the average fitness of the population is
\begin{equation}
\bar{f}(p)=\frac{1}{P_{\rm tot}}\sum_{j=1}^{n}p_jf_j(p_j).\label{af}
\end{equation}
According to \eqref{cp}, the population-share associated with the agent whose fitness is above the average fitness tends to increase, while the population-share associated with the agent whose fitness is below the average fitness tends to decrease. Thus the total population stays constant.

%%%%%%%%%%%%%%%%%%%%%%	      Network Model of DRD	%%%%%%%%%%%%%%%%%%%%%%%%%%%%%%%%%%%%
\subsection{Network model of distributed replicator dynamic} \label{3b}
%%%%%%%%%%%%%%%%%%%%%%%%%%%%%%%%%%%%%%%% figure 1  %%%%%%%%%%%%%%%%%%%%%%%%%%%%%%%%%%%%%
\begin{figure}
	\centering
	\includegraphics[width=2.5in]{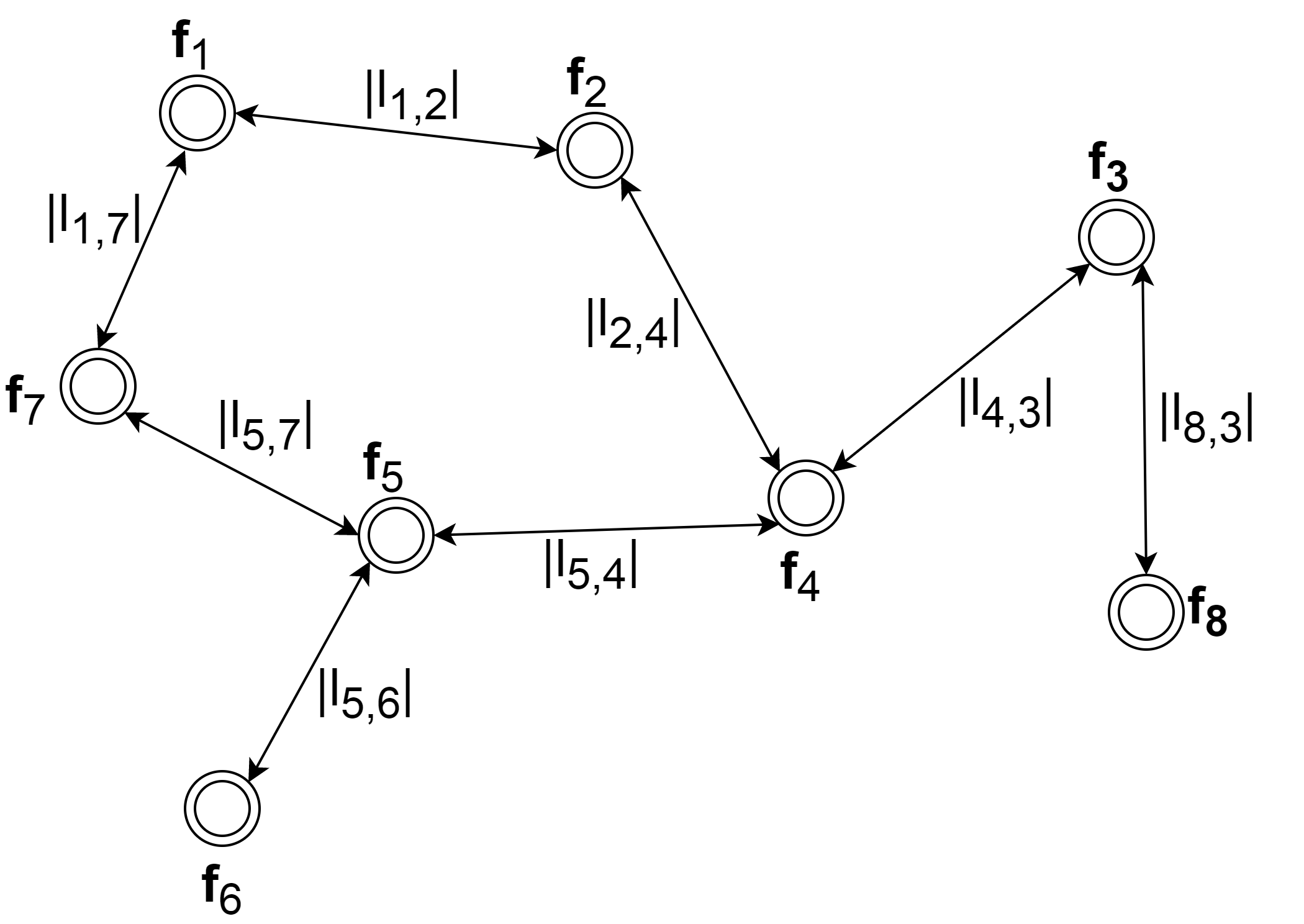}
	\caption{An example of the interaction graph $\mathcal{G}(\mathcal{N},{\mathcal{E}})$ of a finite population of the potential game $f$.}
	\label{graph1}
\end{figure}
Let an undirected graph $\mathcal{G}=\{\mathcal{N},{\mathcal{E}}\}$ depict a network of agents as shown in the example of Fig. \ref{graph1} such that each agent has a nonzero population-share at $t = 0$, where $\mathcal{N}$ denotes the set of agents and $\mathcal{E}$ denotes the set of edges of the graph. Each agent in the network is characterized by the population-share associated with it, which receives a certain fitness for choosing a certain strategy. Henceforth, we allow the fitness scalar to represent the agent wherever necessary. The distributed version of the replicator dynamic \eqref{cp} for the agent is given by\cite{bravo2015distributed,barreiro2016distributed} 
\begin{equation}
\dot{p}_i = \sum_{j\in \mathcal{N}_i} p_i p_j [f_i(p) - f_j(p)],					\label{lre}
\end{equation}
where the difference of fitness between adjacent agents enters into play. 

Note that the simplex introduced in \eqref{simplex} remains invariant under \eqref{cp} and \eqref{lre}.

Introducing the Laplacian matrix $L(p)$ of the graph, i.e.
\begin{equation}
L(p)=\begin{cases}
l_{i,j}=-p_ip_j, \forall (i,j) \in \mathcal{E},\\
l_{i,j}=0, \forall (i,j) \not\in \mathcal{E},\\
l_{i,i}=-\sum_{j\in \mathcal{N}_i}l_{i,j}, \label{aii}
\end{cases}
\end{equation}
where $|l_{i,j}|$ is treated as the weight of the edge connecting $i^{\mathrm{th}}$ and $j^{\mathrm{th}}$ agents on the graph $\mathcal{G}$. Below we write the dynamic \eqref{lre} in a stacked form as
\begin{eqnarray}
\dot{p} = L(p)f(p) \label{nw}
\end{eqnarray}
which represents the replicator dynamic of the population state $p(t)$.
In view of its definition, the Laplacian $L(p)$ is a parameter-varying, real and symmetric matrix which is differentiable and uniformly continuous in $p$. As a consequence, the following holds. 
\begin{statement}\label{as1}
	There exists $\Lambda > 0$ such that the spectral norm $\| L(p) \| < \Lambda,\forall p \in \bigtriangleup,\forall t \geq 0$.
\end{statement}
\begin{statement}\label{as2}
	The gradient of $L(p)$ with respect to $p$ is bounded above by some scalar $\eta$,
	%	\begin{equation*}
	$\left\| \nabla L (p) \right\| \leq \eta, p\in \bigtriangleup.$ %\label{3p}
	%	\end{equation*}
\end{statement}

%%%%%%%%%%%%%%%%%%%%%%	     The Fitness Agreement Protocol %%%%%%%%%%%%%%%%%%%%%%%%%%%%%
\subsection{Fitness agreement protocol}\label{sec3.c}
The fitness dynamic of the population is derived from the dynamic \eqref{nw} as shown below,
%\begin{subequations}
\begin{align*}
\dot{f}(p)&=\nabla f(p)\dot{p},					 \\
&=\nabla^2 V(p)\dot{p},\\ 			%\numberthis 	\label{8b}\\
&= - H L(p)f(p),\\
&= - M(p)f(p),	\numberthis		\label{aut}
\end{align*}
%\end{subequations}  
where $M(p) = H L(p)$ is a positive semidefinite matrix, i.e. $M(p)\geq 0$ [see Theorem 2.2,\cite{wu1988products}]. Notice that $L(p)\boldsymbol{1_n}=0$ which implies $M(p)\boldsymbol{1_n}=0$, where $\boldsymbol{1_n}\in \Re^{n\times 1}$ is a vector containing all ones. That being the case the dynamic \eqref{aut} attains a steady state exactly when all fitnesses are equal. Hence the fitness dynamic \eqref{aut} resembles an agreement protocol\cite{mesbahi2010graph}. 
%%%%%%%%%%%%%%%%%%%%%%%%%%%%%%%%%%%%%%%%%%%%%%%%%%%%%%%%%%%%%%%%%%%%%%%%%%%%%%%%%%%%%%%%
%%%%%%%%%%%%%%%%%%%%%% Existence and Uniqueness of Solutions %%%%%%%%%%%%%%%%%%%%%%%%%%%
%%%%%%%%%%%%%%%%%%%%%%%%%%%%%%%%%%%%%%%%%%%%%%%%%%%%%%%%%%%%%%%%%%%%%%%%%%%%%%%%%%%%%%%%
\subsection{Existence and uniqueness of solution to fitness agreement protocol}\label{sec4}
There exists a unique solution to the dynamic \eqref{nw}\cite{sandholm2010population}. The Hessian matrix $H$ is bounded because $V(p)$ is continuously differentiable and $L(p)$ is also bounded. Consequently, the fitness agreement protocol \eqref{aut} has a unique solution. 
%%%%%%%%%%%%%%%%%%%%%%%%%%%%%%%%%%%%%%%%%%%%%%%%%%%%%%%%%%%%%%%%%%%%%%%%%%%%%%%%%%%%%%%%
%%%%%%%%%%%%%%%%%%%%%% Convergence to NE and its Asymptotic Stability %%%%%%%%%%%%%%%%%%
%%%%%%%%%%%%%%%%%%%%%%%%%%%%%%%%%%%%%%%%%%%%%%%%%%%%%%%%%%%%%%%%%%%%%%%%%%%%%%%%%%%%%%%%
\subsection{Convergence of the DRD to the NE}\label{sect5} 															
The definition \eqref{nef} straightforwardly extends to potential games, however, for such games the NE is unique seeing that the potential function $V(p)$ is strictly quadratic concave. 

The convergence of the population state $p(t)$ to the NE of the potential game may also be described from the graph-theoretic perspectives. In what follows, we provide our main results that link the graph theory with the population games and show that the existence of the NE $p^*$ is strongly associated with the connectedness of the interaction graph $\mathcal{G}$.
%%%%%%%%%%%%%%%%%%%% change citation to following def n lemma %%%%%%%%%%%%%%%%%%%%%%%%%%
%%%%%%%%%%%%%%%%%%%%%%%%%%%%%%%%%%%%%%%% Main Results %%%%%%%%%%%%%%%%%%%%%%%%%%%%%%%%%%%
\subsection{Main results: Part A} \label{sec5.a}
%%%%%%%%%%%%%%%%%%%%%%%%%%%%%%%%%%%%%%%% lemma 5.1  %%%%%%%%%%%%%%%%%%%%%%%%%%%%%%%%%%%%%
\begin{lemma}\label{5.1}
	In the potential game $f$, all agents receive equal fitness if and only if $\lambda_2(L(p(t)))>0,\forall t \geq 0$.
\end{lemma}
\begin{proof}
	Since the edge-weights $|l_{i,j}|$ are dynamic, for the graph $\mathcal{G}$ to stay connected, one must ensure that $l_{i,j}<0,\forall i,j \in \mathcal{E}, \forall t\geq 0$. This condition can also be stated directly in terms of the second smallest eigenvalue of $L(p(t))$\cite{mesbahi2010graph}. Assuming that $\mathcal{G}$ is initially connected, the necessary and sufficient condition for $\mathcal{G}$ to stay connected throughout the game play is that $\lambda_2(L(p(t))) >0,\forall t \geq 0$. $\lambda_2(L(p(t)))>0,\forall t \geq 0$ ensures that $\lambda_2(M(p(t)))>0,\forall t \geq 0$ since $H>0$. Hence,  $\lambda_2(L(p(t)))>0,\forall t \geq 0$ is necessary and sufficient to prove that the agreement protocol \eqref{aut} converges to the agreement value.
\end{proof}

\indent In line with the previous Lemma \ref{5.1}, we derive the next result on the convergence of the dynamic given in \eqref{nw} to the NE of the game.
%%%%%%%%%%%%%%%%%%%%%%%%%%%%%%%%%%%%%%%% theorem 5.2  %%%%%%%%%%%%%%%%%%%%%%%%%%%%%%%%%%%%%
\begin{theorem}\label{5.2}
	If $\lambda_2(L(p(t)))>0,\forall t \geq 0$ then the fixed point of the dynamic \eqref{nw} coincides with the NE $p^{*}$.
\end{theorem}
\begin{proof}
	Note $\lambda_2(L(p(t)))>0,\forall t \geq 0$ implies that all edges have non-zero weights for all time. Since $l_{i,j}=-p_i(t)p_j(t)$, it further implies that the population-shares associated with all agents remain non-zero (i.e $p_i(t)>0, \forall t>0$). Secondly, as a consequence of Lemma \ref{5.1} and the fact that $M(p)\boldsymbol{1_n}=0$, the dynamic \eqref{nw} reaches a fixed point, say $p^s$ that qualifies to be the NE $p^*$ of the potential game, as mentioned in Definition \eqref{nef}. Hence proved.
\end{proof}

\indent Lemma \ref{5.1} and Theorem \ref{5.2} provide the next corollary.
%%%%%%%%%%%%%%%%%%%%%%%%%%%%%%%%%%%%%%%% corollary 5.2.1 %%%%%%%%%%%%%%%%%%%%%%%%%%%%%%%%
\begin{corollary}\label{cor}
	The condition $\lambda_2(L(p(t)))>0,\forall t \geq 0$ is a sufficient condition for the fixed point $p^s$ to be the NE $p^*$ of the potential game $f$.
\end{corollary}
A direct consequence of Corollary \ref{cor} provides the following result.
%%%%%%%%%%%%%%%%%%%%%%%%%%%%%%%%%%%%%%%% lemma 5.3  %%%%%%%%%%%%%%%%%%%%%%%%%%%%%%%%%%%%%
\begin{theorem} \label{5.3}
	Let 
	\begin{equation*}
	\sigma^0=\{\lambda_k(L(p^s))= 0, \forall k=2,\ldots,n\} 
	\end{equation*} for $p^s \not=p^*$, then $\sigma^0\not=\emptyset$ indicates the existence of singleton graphs. 
\end{theorem}
\begin{proof}
	It is obvious that $|\sigma^0|+1$ is the algebraic multiplicity of the eigenvalue $\lambda_1(p)=0$, which is equal to the number of connected components in  $\mathcal{G}$. If $\lim_{t \rightarrow \infty}l_{i,j} = 0$ for some $(i,j) \in \mathcal{E}$ then there must be either $\lim_{t \rightarrow \infty}p_i=0$ or $\lim_{t \rightarrow \infty}p_j=0$, or both; which indicates that the population-share associated with either agent $i$ or agent $j$, or both $i^{\mathrm{th}}$ and $j^{\mathrm{th}}$ agents tend to decline. Let $\mathcal{K}$ be a set of such agents then there must be $|\mathcal{K}|\leq |\sigma^0|$ number of isolated agents in the graph $\mathcal{G}$. These isolated agents represent singleton graphs. Hence proved.
\end{proof}
\subsection{Observations on truncation behavior of agents:}
\indent This subsection is in line with the  truncation behavior discussed in \cite{pantoja2011dispatch}. Here we introduce the singleton graphs which lie at the center of the truncation phenomenon. Let $\mathcal{G}^c \in \mathcal{K},\forall c=1,\ldots,|\mathcal{K}|$ be singleton graphs as defined earlier.
%%%%%%%%%%%%%%%%%%%%%%%%%%%%%%%% case 1 %%%%%%%%%%%%%%%%%%%%%%%%%%%%%%%%%
If $f_i(p^s)$ is a fitness corresponding to the population-share $p^s_i=0$ for some $i \in \mathcal{N}$ then we have two following cases. 

\subsubsection*{Case 3.1. When agent $f_i(p^s)$ is an end-point of $\mathcal{G}(\mathcal{N},{\mathcal{E}})$}
\label{case1}
The population-share $p^s_i=0$ implies that eventually $l_{i,j}=0, \forall j\in \mathcal{N}_i$ or $\mathcal{N}_i=\emptyset$, following which $\mathcal{G}_c=\{f_i(p^s)\}$. Since $f_i(p^s)$ is being the end-point in $\mathcal{G}(\mathcal{N},{\mathcal{E}})$, its isolation renders the connectedness of the remaining part of $\mathcal{G}$ unaffected i.e., as shown in the example of Fig. \ref{graph2}.
%%%%%%%%%%%%%%%%%%%%%%%%%%%%%%%%%%%%%%%% figure 2  %%%%%%%%%%%%%%%%%%%%%%%%%%%%%%%%%%%%%
\begin{figure}
	\centering
	\includegraphics[width=2.65in]{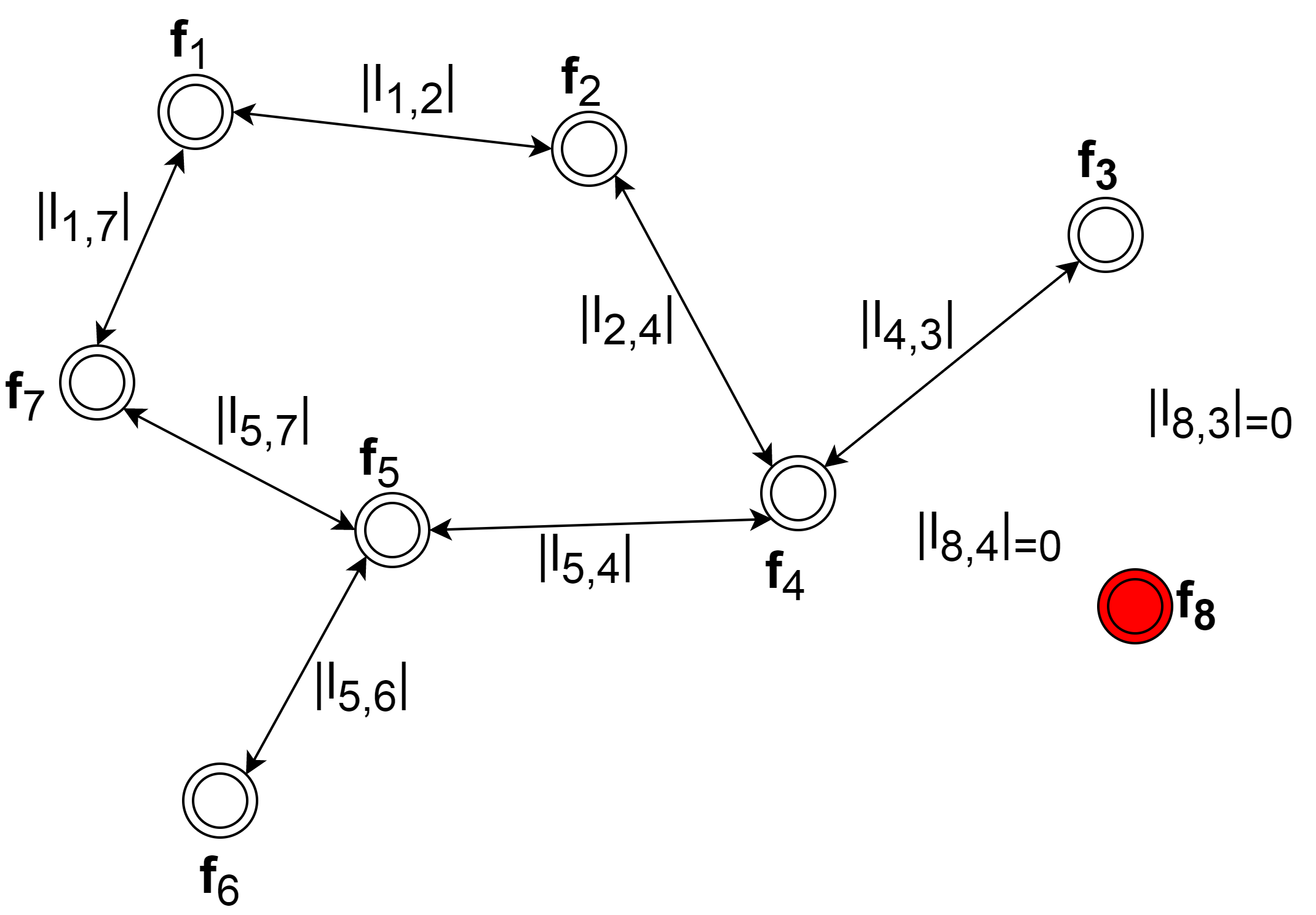}
	\caption{Interaction graph $\mathcal{G}(\mathcal{N},\mathcal{E})$ from Fig. \ref{graph1} depicting Case 3.1: Agent $f_8$ is an end-point.}
	\label{graph2}
\end{figure}
%%%%%%%%%%%%%%%%%%%%%%%%%%%%%%%%%%%%%%%% figure 3  %%%%%%%%%%%%%%%%%%%%%%%%%%%%%%%%%%%%%
\begin{figure}
	\centering
	\includegraphics[width=2.5in]{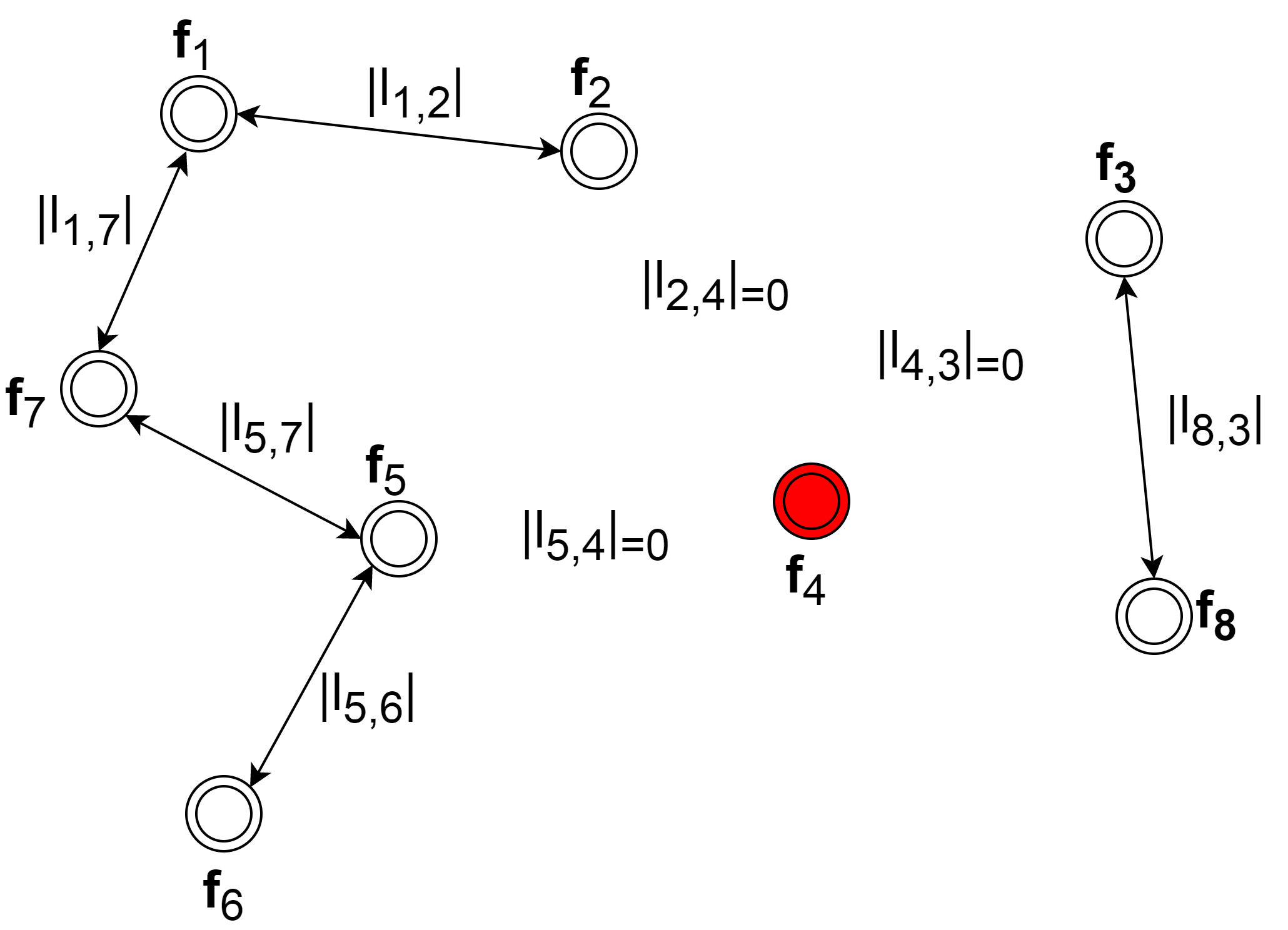}
	\caption{Interaction graph $\mathcal{G}(\mathcal{N},\mathcal{E})$ from Fig. \ref{graph1} depicting Case 3.2: Agent $f_4$ is a cut-point.}
	\label{graph3}
\end{figure}
%%%%%%%%%%%%%%%%%%%%%%%%%%%%%%%%%%%%%%%% case 2  %%%%%%%%%%%%%%%%%%%%%%%%%%%%%%%%%%%%%
\subsubsection*{Case 3.2. When agent $f_i(p^s)$ is the cut-point of $\mathcal{G}(\mathcal{N},{\mathcal{E}})$}
This scenario is depicted in Fig. \ref{graph3}, if the agent $i$ holds the position in $\mathcal{G}$ such that $\mathcal{G}' \cap \mathcal{G}''=\{f_i(p^s)\}$ for any two components, say $\mathcal{G'}$, and $\mathcal{G''}$ in $\mathcal{G}$; then the edges that have isolated $f_i(p^s)$ also separate $\mathcal{G'}$ and $\mathcal{G''}$ from each other.
%\end{proof}
%%%%%%%%%%%%%%%%%%%%%%%%%%%%%%%%%%%%%%%% Stability Analysis %%%%%%%%%%%%%%%%%%%%%%%%%%%%%%
\subsection{Stability analysis of dynamics \eqref{nw} and \eqref{aut}}\label{sec5.b}
\indent Here we show that the stability of the fitness agreement protocol \eqref{aut} and the dynamic \eqref{nw} share a common platform. Before proceeding further, it is worth considering certain notions of the dynamic \eqref{nw} \cite{sandholm2015population} stated subsequently.
\begin{enumerate}
	\item Positive Correlation (PC): 
	\begin{equation}
	\dot{p}(t)\neq0 \implies \dot{p}(t)^Tf(p)>0. \label{cond5.1}
	\end{equation}
	\item Nash Stationarity (NS): 
	\begin{equation}
	\dot{p}(t)=0 \Longleftrightarrow p \in NE(f). \label{cond5.2}
	\end{equation}
\end{enumerate}
\indent Conditions \eqref{cond5.1} and \eqref{cond5.2} strongly imply that the potential function $V(p)$ attains the maximum at $p^*$. The asymptotic stability of the NE $p^*$ of the potential game is bound by these conditions as stated in \cite{sandholm2015population}. 
A similar approach has been considered for the asymptotic stability of the distributed replicator dynamic in \cite{barreiro2016distributed}. The analysis considers that the graph is connected. In essence, the stability analysis is restricted to only $p^*$. However, it does not imply that the distributed replicator dynamic has an unstable equilibrium state when the graph loses the connectivity. In this scenario, it is more convincing to analyze the stability of the fitness agreement protocol \eqref{aut} instead of the dynamic \eqref{nw}. In what follows, we intend to prove that the stability of the fitness agreement protocol \eqref{aut} implies the stability of the dynamic \eqref{nw}.
\subsubsection{Main results: Part B}
\begin{theorem}\label{thm5.6}
	Let $W(p):\Re^n \times \Re \rightarrow \Re_{\geq0}$ be a candidate Lyapunov function chosen such that $\dot{W}(p)\leq -\lambda_2(L(p))f^T(p)f(p)$ then the fitness agreement protocol \eqref{aut} is asymptotically stable.
\end{theorem}
\begin{proof} Since $H$ is a symmetric positive definite matrix, it is invertible. Let the candidate Lyapunov function for the system \eqref{aut} be,
	\begin{equation}
	W(p)=\frac{1}{2}f^T(p)H^{-1}f(p).\label{vdot}
	\end{equation}
	Differentiating \eqref{vdot} along the trajectories of the fitness vector yields,
	\begin{subequations}
		\begin{eqnarray}
		\dot{W}(p)=-f^T(p)L(p)f(p). \label{vdot1}
		\end{eqnarray}
	\end{subequations}
	Now, if the graph $\mathcal{G}$ remains connected $\forall t \geq 0$, then we have \begin{equation*}
	f^T(p)L(p)f(p) \geq \lambda_2(L(p))f^T(p)f(p).
	\end{equation*} 
	Using this inequality in \eqref{vdot1}, we get
	\begin{eqnarray}
	\dot{W}(p)\leq-\lambda_2(L(p))f^T(p)f(p).\label{lvdot}
	\end{eqnarray}
	Hence proved. 
\end{proof}

\begin{theorem}\label{3.2}
	The asymptotic stability of the fitness dynamic \eqref{aut} implies the asymptotic stability of the dynamic \eqref{nw} and vice versa.
\end{theorem}
\begin{proof}
	The proof simply follows from the equation \eqref{vdot1}. We have,
	\begin{subequations}
		\begin{eqnarray}
		\dot{W}(p)&=&-f^T(p)\{L(p)f(p)\},\\
		&=&-f^T(p)\dot{p},\label{rela}\\
		&\leq&0.\label{inv}
		\end{eqnarray}
	\end{subequations}
	Comparing the equation \eqref{rela} to the inequality \eqref{cond5.2}, we assert that the positive correlation holds between $f^t(p)$ and $\dot{p}$, and along with the condition $\lambda_2(L(p(t)))>0,\forall t \geq 0$, the Nash stationarity also holds. Hence, the dynamic \eqref{nw} is also asymptotically stable. Hence proved. 
\end{proof}
\begin{corollary}\label{cor3.8}
	Every connected component in $\mathcal{G}$ exhibits asymptotically stable dynamics.
\end{corollary}
\begin{proof}
	If $p^s\not=p^*$ then there exist $|\sigma^0|$ connected components in $\mathcal{G}$ and each connected component, i.e. $\mathcal{G}_q,~\forall q=1,\ldots,|\sigma^0|$ has an associated Laplacian $L_q(p)$; owing to which the graph-Laplacian $L(p)$ assumes a block diagonal form:
	\begin{equation}
	L(p)=\left(\begin{array}{ccccc}
	L_1(p) & \hfill & \hfill & \hfill & \hfill\\
	\hfill & L_2(p) & \hfill & \hfill & \hfill\\
	\hfill & \hfill &  \ddots & \hfill & \hfill\\
	\hfill & \hfill & \hfill & L_{|\sigma^0|-1}(p) & \hfill\\
	\hfill & \hfill & \hfill & \hfill & L_{|\sigma^0|}(p)\\
	\end{array}\right).
	\end{equation} 
	\indent Apparently, each connected component has a $0$ eigenvalue of algebraic multiplicity equal to 1. Also, each $\mathcal{G}$ has a subset of agents having their population and fitness dynamics. In line with this, the stability analyses discussed in Theorems \ref{thm5.6} and \ref{3.2} readily apply to the connected components in $\mathcal{G}$.
\end{proof}
\section{Simulation results}\label{sim}
All simulations are performed using the {MATLAB-SIMULINK} software. The population graphs are plotted using ``wgPlot.m" and ``gplotwl.m" matlab routines. We first consider a population dynamic of 5 agents. A simple nearest-neighbor interconnection is considered as shown in Fig. \ref{graph9}. Initially, the total population is equally distributed among all agents ($p_i(0)=0.2,\forall i \in \mathcal{N}$), owing to which all the edges initially have equal weights as indicated by the blue color associated with them. The vertical color map indicates the weight associated with each agent i.e., the value of the fitness scalar which is also correlated with the size of its dot. Initially, the fitness scalar have been assigned different weights. As fitness agreement protocol evolves, all fitness scalars except $f_5$ converge to the agreement value as can be seen from Fig. \ref{graph4}. It is due to the fact that the population state $p_5$ belonging to the $5^{\mathrm{th}}$ agent tends to decline and eventually goes to 0 as shown in Fig. \ref{graph5}. This results in removal of the edges $|l_{1,5}|$ and $|l_{4,5}|$ as reflected in Fig. \ref{graph7}. As a consequence of the this, at a steady state, the graph $\mathcal{G}$ comprises two connected components, i.e. the one that retains the agents $f_1$ to $f_4$ and the other one being singleton graph that contains $f_5$. Fig. \ref{graph8} confirms the existence of 2 connected components in $\mathcal{G}$. The population structure at the steady state is manifested in Fig. \ref{graph10} wherein the edge-weights are associated with specific color shades on the horizontal color map whose numeric values are also indicated alongside the respective edges. Fig. \ref{graph6} verifies that the simplex \eqref{simplex} remains invariant.

In a similar way, we also consider a random network of 20 agents as shown in Fig. \ref{graph16}. Initially, the total population is equally distributed among all agents and the fitness scalars are calculated to have different values. %Fig. \ref{graph11} depicts that not all agents convergence to the agreement value $\bar{f}(p)$ because of some population states eventually going to 0 as seen in Fig. \ref{graph12}. 
The connected components in $\mathcal{G}$ are highlighted by boxed regions in Fig. \ref{graph17}, the graph comprises 2 connected components and 9 isolated agents. 
\begin{figure}[t!]
	\centering
	\includegraphics[width=3.65in]{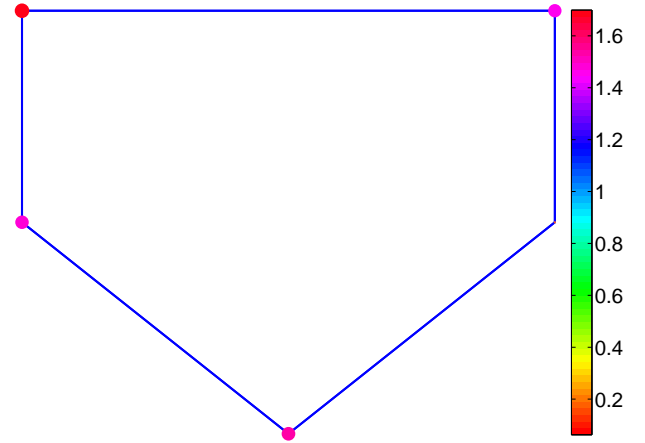}
	\caption{The graph $\mathcal{G}$ consisting 5 agents at time $t_0$.}
	\label{graph9}
\end{figure}
\begin{figure}[t!]
	\centering
	\includegraphics[width=3.65in]{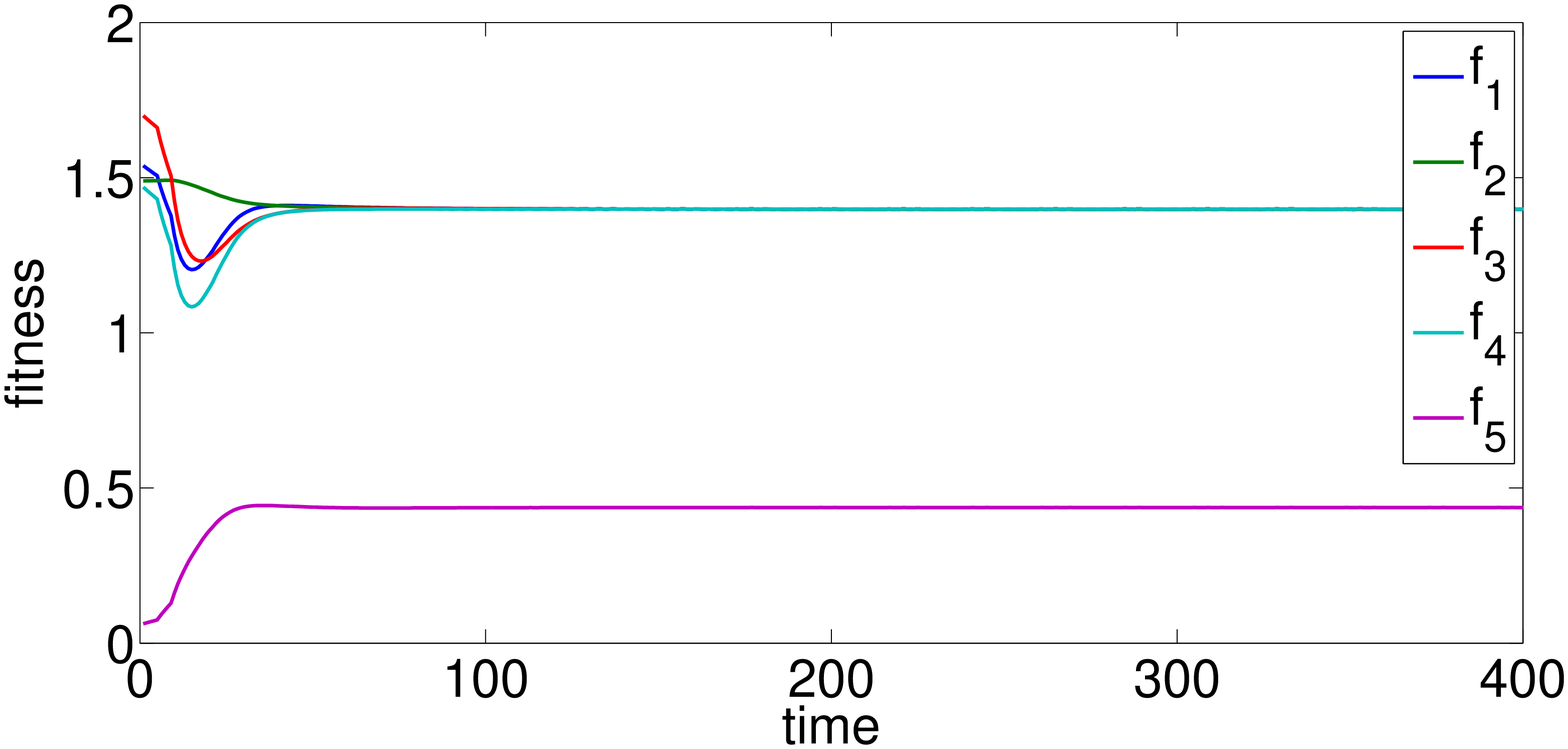}
	\caption{Evolution of fitness of the agents on the graph $\mathcal{G}$ from Fig. \ref{graph9}.}
	\label{graph4}
\end{figure}
\begin{figure}[t!]
	\centering
	\includegraphics[width=3.65in]{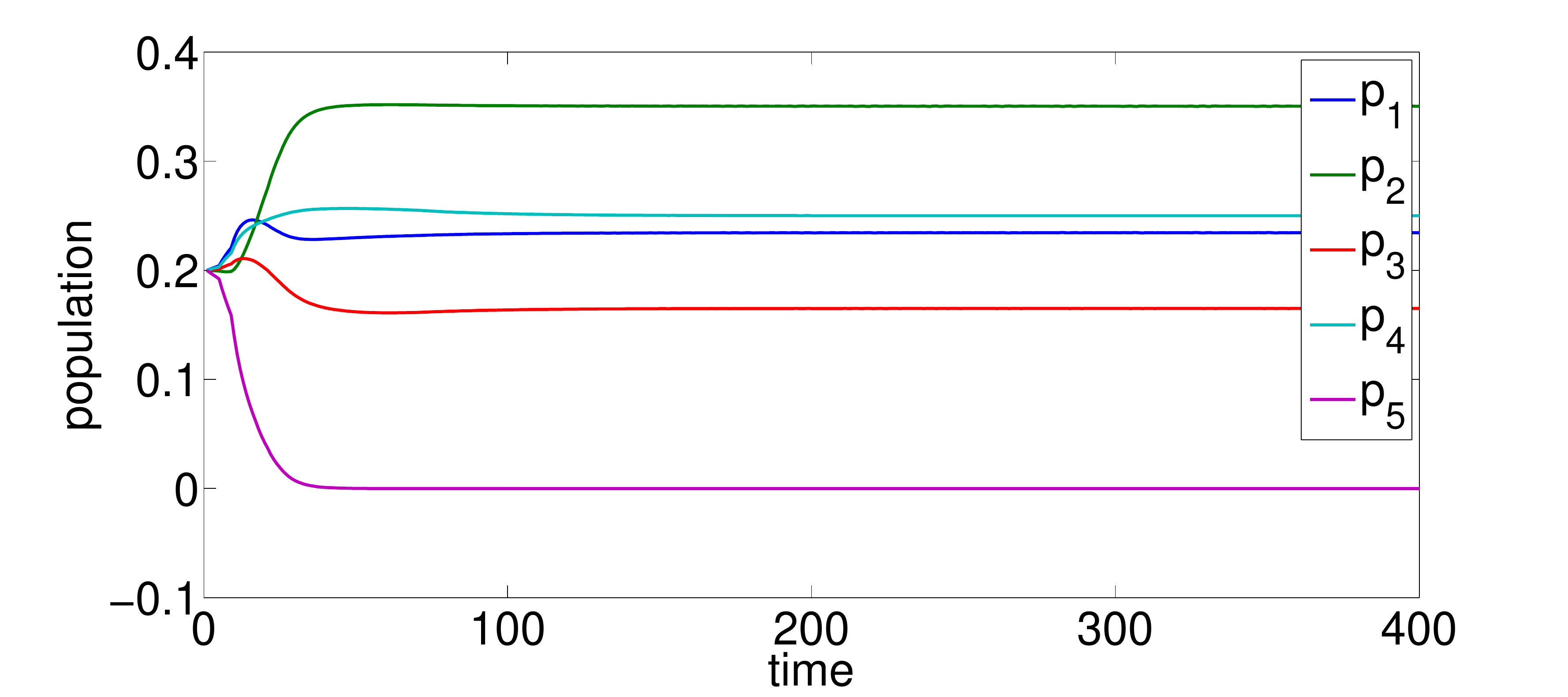}	
	\caption{Evolution of population of the agents on the graph $\mathcal{G}$ from Fig. \ref{graph9}.}
	\label{graph5}
\end{figure}
\begin{figure}[t!]
	\centering
	\includegraphics[width=3.65in]{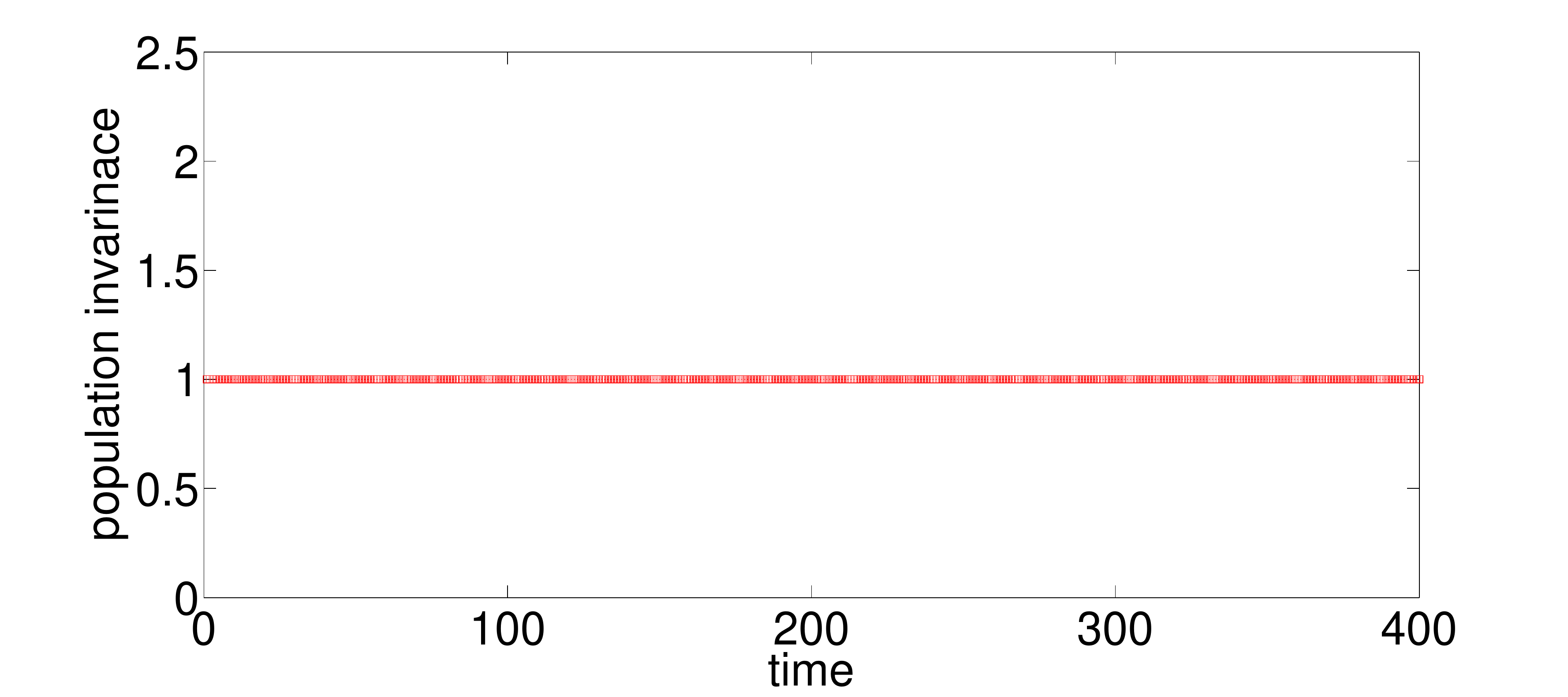}
	\caption{Invariance of population of the agents on the graph $\mathcal{G}$ from Fig. \ref{graph9}.}
	\label{graph6}
\end{figure}
\begin{figure}[t!]
	\centering
	\includegraphics[width=3.65in]{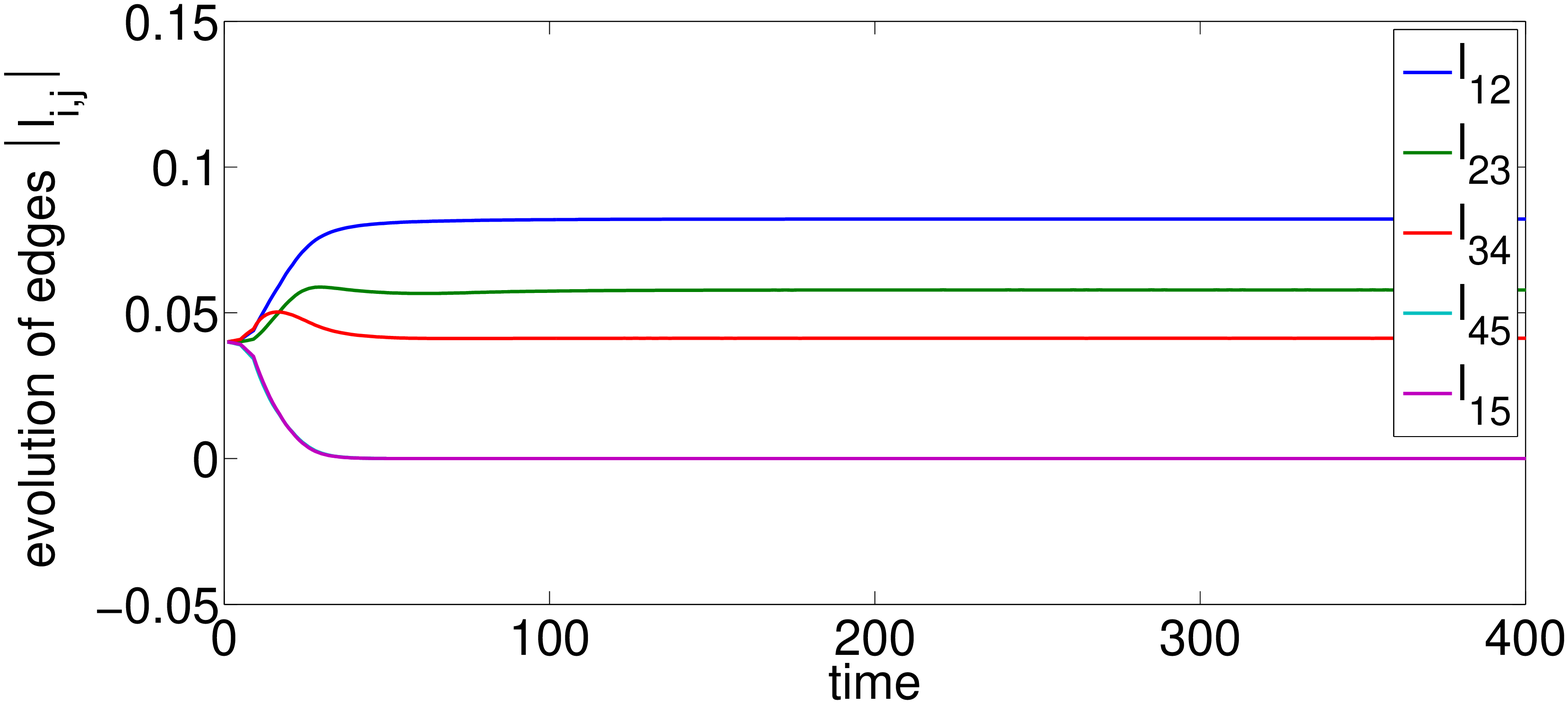}
	\caption{Evolution of the edges of the graph $\mathcal{G}$ from Fig. \ref{graph9}.}
	\label{graph7}
\end{figure}
\begin{figure}[t!]
	\centering
	\includegraphics[width=3.65in]{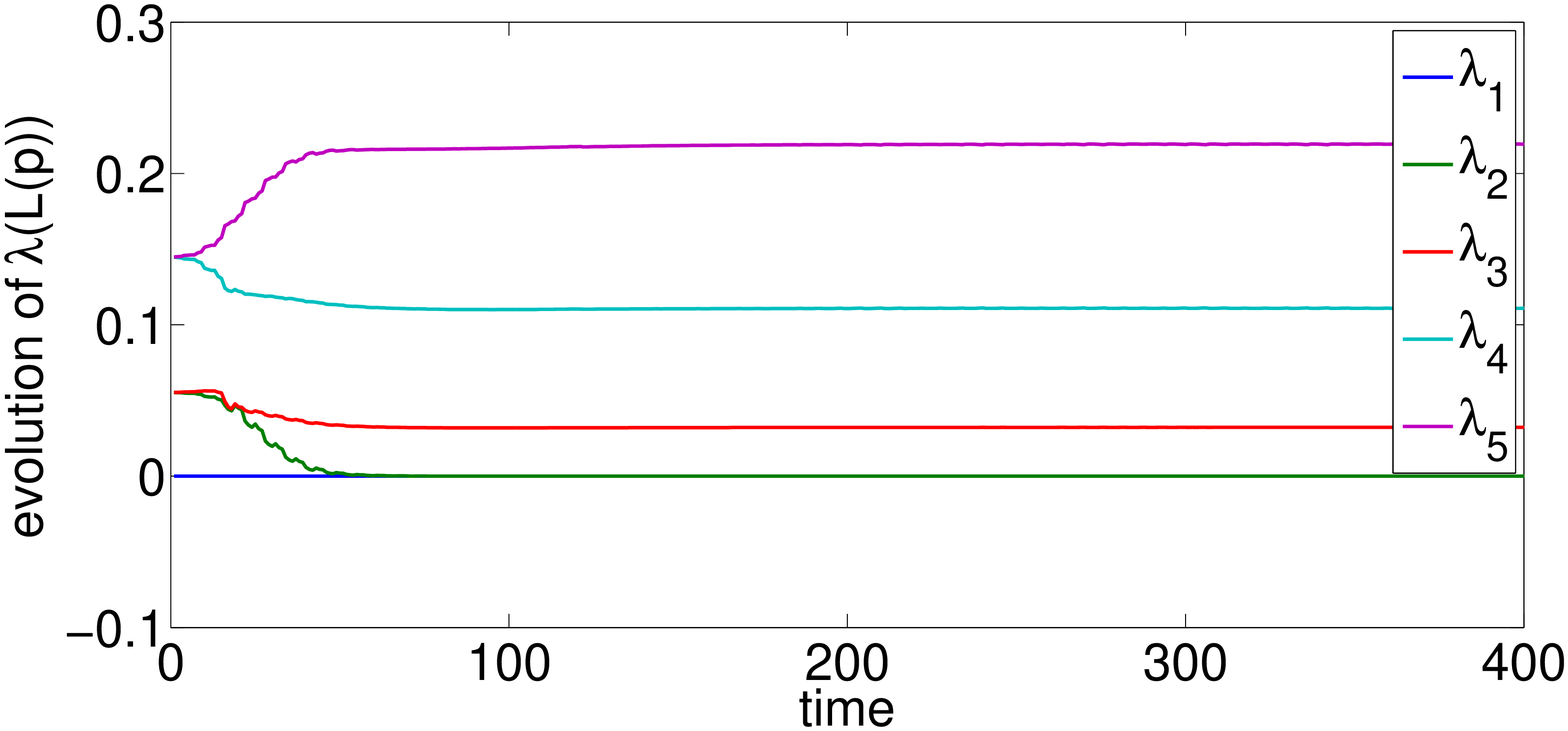}
	\caption{Evolution of eigenvalues of the Laplacian $L(p)$ of the graph $\mathcal{G}$ from Fig. \ref{graph9}.}
	\label{graph8}
\end{figure}

\begin{figure}[t!]
	\centering
	\includegraphics[width=3.65in]{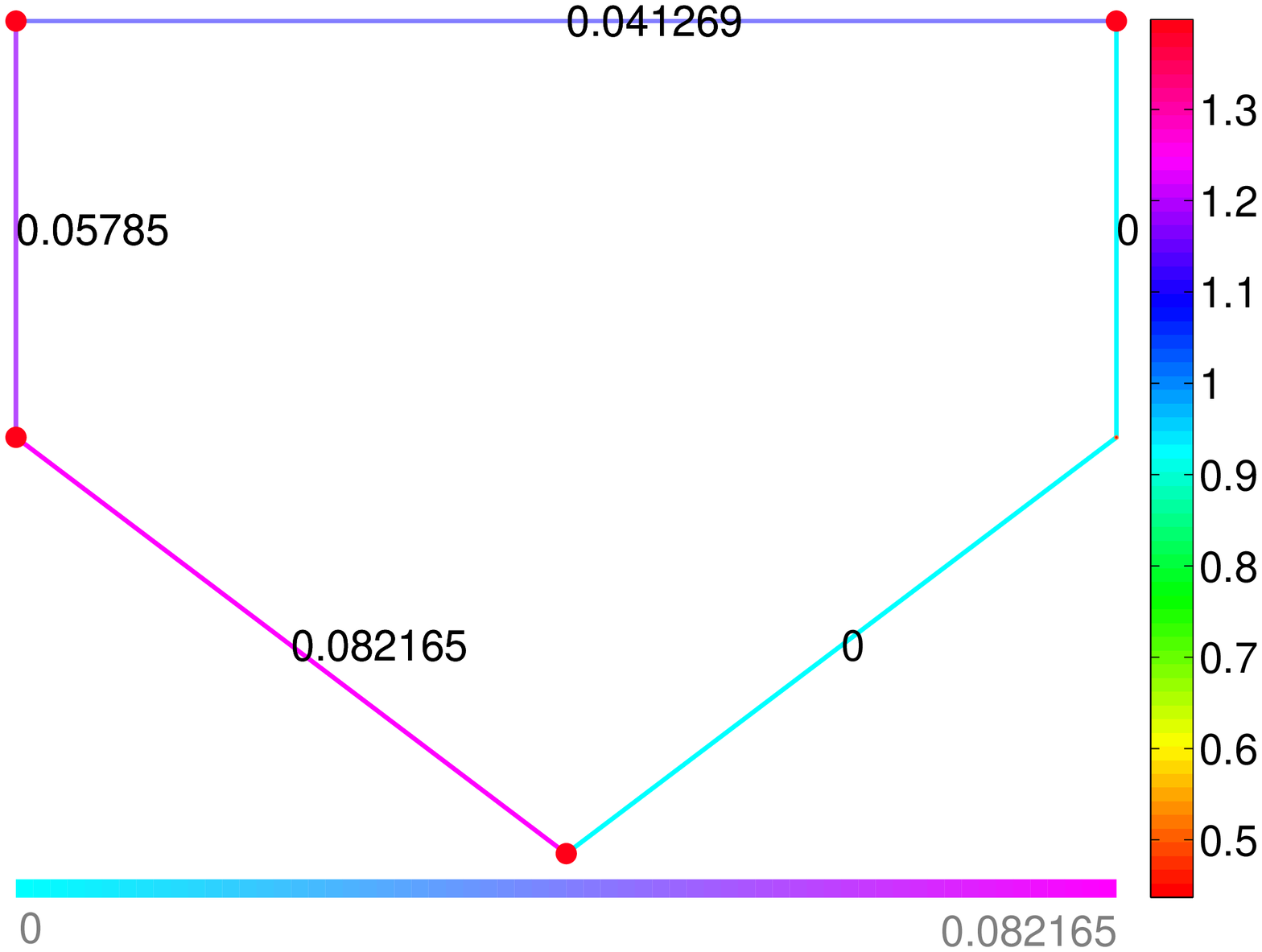}
	\caption{The graph $\mathcal{G}$ from Fig. \ref{graph9} at steady state.}
	\label{graph10}
\end{figure}

\begin{figure}[t!]
	\centering
	\includegraphics[width=3.65in]{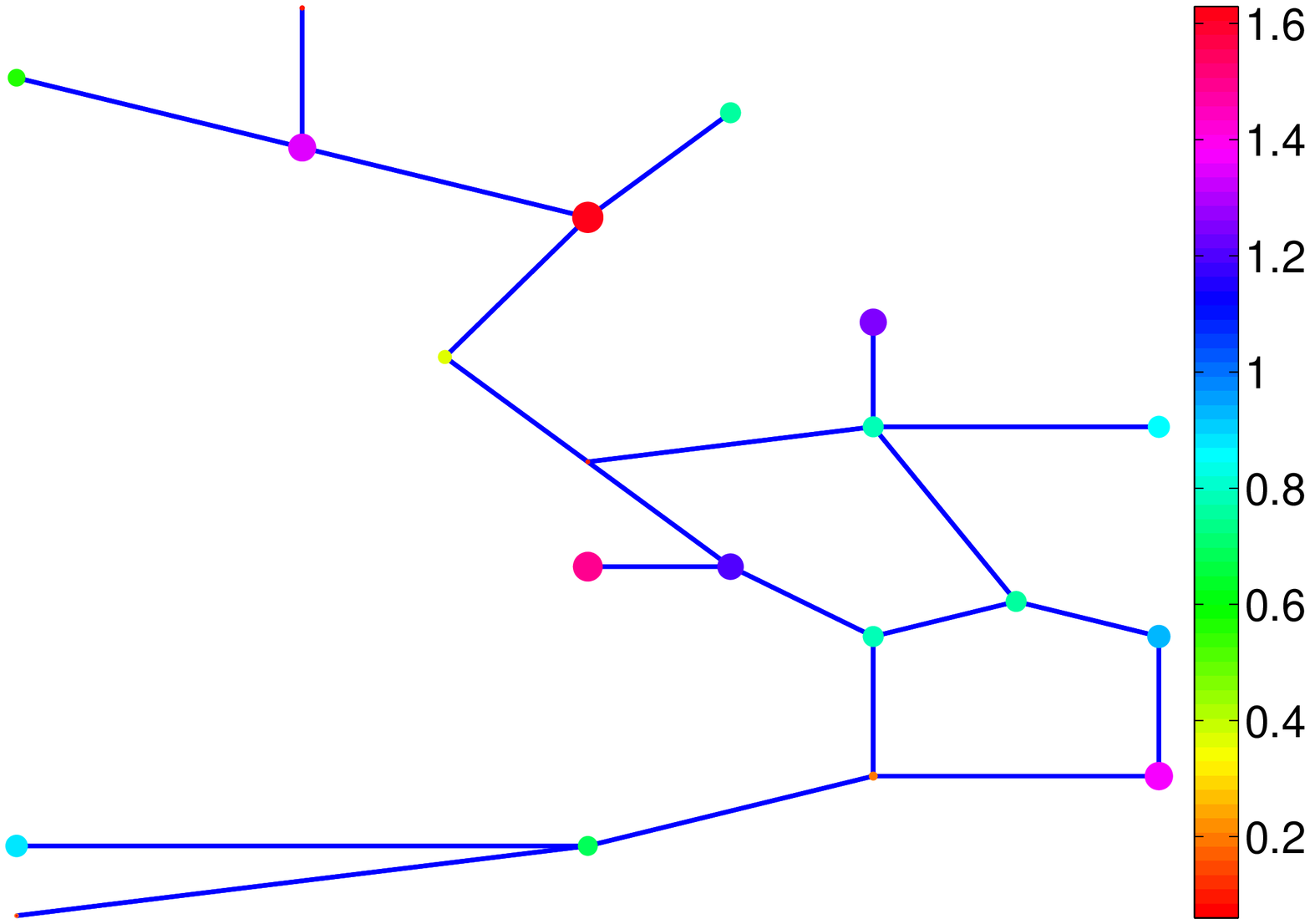}
	\caption{The graph $\mathcal{G}$ consisting 20 agents at time $t_0$.}
	\label{graph16}
\end{figure}
\begin{figure}[t!]
	\centering
	\includegraphics[width=3.65in]{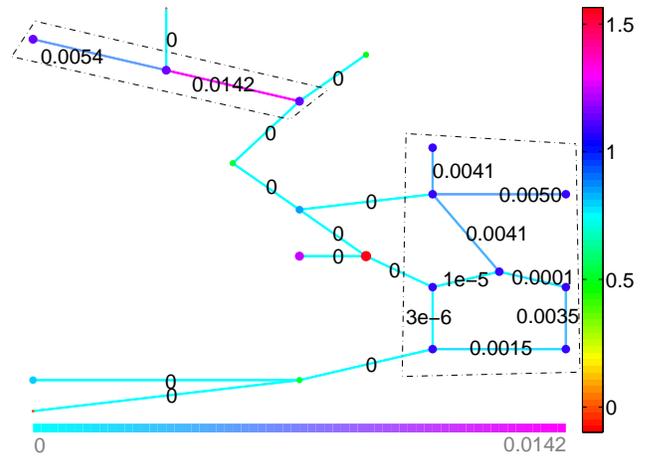}
	\caption{The graph $\mathcal{G}$ from Fig. \ref{graph16} at steady state.}
	\label{graph17}
\end{figure}
\section{Conclusions}\label{con}
A macroscopic behavior of the distributed replicator dynamic is analyzed using graph theory. The fitness dynamic is shown to resemble a linear parameter-varying dynamic agreement protocol. A comprehensive analysis of truncation behavior in the distributed replicator dynamic is put forth using spectral graph theory. It is proved that the existence and asymptotic stability of the Nash equilibrium of the population games are dependent on the eigenvalue spectrum of the graph-Laplacian matrix of the underlying network. Also, it is shown that with a right choice of a Lyapunov function for the fitness dynamic; the asymptotic stability of the fitness dynamic implies the asymptotic stability of the distributed replicator dynamic.
 	
\bibliographystyle{unsrt}  
\bibliography{tran_updated} 
\end{document}